\documentclass[reqno,11pt]{amsart}
\usepackage[margin=1.1in, footskip=1cm]{geometry}
\usepackage{amsmath, amsthm, amssymb, amsfonts, booktabs}
\usepackage{mathrsfs}
\usepackage{thmtools, thm-restate}
\usepackage{xcolor}
\allowdisplaybreaks
\pdfpagewidth=\paperwidth
\pdfpageheight=\paperheight
\usepackage[initials,nobysame]{amsrefs}

\newtheorem{theorem}{Theorem}[section]
\newtheorem{lemma}[theorem]{Lemma}
\newtheorem{corollary}[theorem]{Corollary}

\theoremstyle{definition}

\theoremstyle{remark}

\numberwithin{equation}{section}

\newcommand{\mmod}[1]{\,\,({\rm{mod}}\,\,#1)}

\def\calM{{\mathcal M}}

\def\calS{{\mathcal S}}

\def\calX{{\mathcal X}}
\def\calY{{\mathcal Y}}
\def\calZ{{\mathcal Z}}

\def\setmin{{\setminus\!\,}}

\def\C{{\mathbb C}}\def\F{{\mathbb F}}
\def\T{{\mathbb T}}
\def\Z{{\mathbb Z}}
\def\K{{\mathbb{K}}}

\def\grc{{\mathfrak c}}

\newcommand{\hA}{\widehat{A}}

\newcommand{\hC}{\widehat{C}}
\newcommand{\hD}{\widehat{D}}

\newcommand{\hH}{\widehat{H}}

\newcommand{\hi}{\hat{i}}

\newcommand{\hK}{\widehat{K}}

\newcommand{\hM}{\widehat{M}}
\newcommand{\hN}{\widehat{N}}

\newcommand{\hR}{\widehat{R}}

\newcommand{\ord}{\text{ord}}
\newcommand{\tr}{\text{tr}}

\renewcommand{\d}{\,{\rm d}} 

\renewcommand{\epsilon}{\varepsilon}
\newcommand{\vrho}{\varrho}

\newcommand{\<}{\begin{equation}}
\renewcommand{\>}{\end{equation}}

\makeatletter
\@namedef{subjclassname@2020}{\textup{2020} Mathematics Subject Classification}
\makeatother

\begin{document}

\title{Subconvexity of Short $k$-Free Exponential Sums in $\F_q[t]$}
\author[Ben Doyle]{Ben Doyle}
\address{Department of Mathematics, Purdue University, 150 N. University Street, West 
Lafayette, IN 47907-2067, USA}
\email{doyle133@purdue.edu}
\thanks{}
\date{}

\begin{abstract}
    We extend recent work of the author over $\Z$ into the positive characteristic setting of $\F_q[t]$. In particular, for a polynomial $F \in \F_q[t]$ of degree $N$, let $R_k(\alpha)$ denote the exponential sum over $k$-free polynomials $f$ with $\deg(f-F)<K$. For all $s>0$, we prove essentially tight upper and lower bounds for the $s$-th moment of $R_k(\alpha)$ whenever $K > (\frac{1}{2}+\epsilon)N$, and in even shorter intervals when $s>1+1/k$. As an application of these results, we prove a lower bound of order $q^{\frac{K}{6}}$ for the $L^1$-mean of the M\"obius-twisted exponential sum over $\F_q[t]$ whenever $K >(\frac{1}{2}+\epsilon)N$. 

    \vspace{-1cm}
\end{abstract}

\maketitle

\section{Introduction}

Let $\mu(n)$ denote the M\"obius function and $e(\alpha) := e^{2\pi i \alpha}$. Define the exponential sum over the squarefree integers by
\[S_2(\alpha;K) = S_2(\alpha) = \sum_{N-K<n \leq N} \mu^2(n)e(n\alpha).\]
In a recent paper \cite{doyle:2026+}, the author established a subconvexity bound for $S_2(\alpha;K)$ in short intervals, improving a result of Sun \cite{sun:2023}.
\begin{theorem}[Doyle, \cite{doyle:2026+}]\label{L1-thm-Z}
    Whenever $N \geq K \gg N^{0.49685}$, one has that
    \[\int_0^1 |S_2(\alpha;K)| \d\alpha \asymp K^{\frac{1}{3}}.\]
\end{theorem}

One may naturally ask if this behavior is observed in the positive-characteristic setting of the polynomial ring $\F_q[t]$. To state our result, we require some notation, which we will describe in more detail in \S2. The $\F_q[t]$-analogue of the real numbers is the field of formal Laurent series $K_\infty = \F_q((1/t))$. For $\alpha \in \K_\infty$, we define $\ord(\alpha)$ to be the degree of the largest non-zero coefficient of $\alpha$, and we define the torus $\T = \{\ord(\alpha)<1\}$. We may then define a natural Haar measure $\d\alpha$ on $K_\infty$ normalized so that $\int_\T \d\alpha = 1$. Let the M\"obius function $\mu_q(f)$ over $\F_q[t]$ be the unique multiplicative function with $\mu_q(\varpi) = -1$ and $\mu_q(\varpi^k) = 0$ for $k>1$ for $\varpi \in \F_q[t]$ irreducible. Finally, we may construct an additive character $\psi:K_\infty \rightarrow \C^\times$ over $K_\infty$ which is periodic modulo $\T$, defined explicitly in \eqref{psi-def-eq}. For $F \in \F_q[t]$ with $\deg(F) = N$, we then define the exponential sum
\<\label{R_k-def-eq}
R_{2,q}(\alpha;K) = R_2(\alpha) = \sum_{\deg(f-F) < K} \mu_q^2(f)\psi(f\alpha).
\>

\begin{theorem}\label{L1-thm-ff}
    Let $F \in \F_q[t]$ have degree $N$, and let $K \leq N$. Then for any $\epsilon>0$ and $N$ sufficiently large, one has that
    \[\int_\T |R_2(\alpha)| \d\alpha \asymp_{q,\epsilon} \hK^{\frac{1}{3}}\]
    as $N \rightarrow \infty$ whenever $K > (\frac{1}{2}+\epsilon)N$, where $\hK = q^K$.
\end{theorem}

We will use the notation $\hA = q^A$ throughout, as it indicates more clearly the comparison between the integer and function field settings. Our results consider the case in which $q$ remains fixed and $N \rightarrow \infty$. It would also be interesting to consider the analogous problem when $q \rightarrow \infty$ and $N$ remains fixed. 

We are able to analyze the behavior of such exponential sums in much wider generality than Theorem \ref{L1-thm-ff}, following a framework introduced by Keil \cite{keil:2013}. For an integer $k \geq 2$, let $\mu_{k,q}(f)$ denote the indicator function of the $k$-free polynomials (that is, polynomials indivisible by any $k$-th power of a polynomial of positive degree) so that $\mu_{2,q}(f) = \mu_q^2(f)$. For $s>0$, a polynomial $F \in \F_q[t]$ of degree $N$ and $K \leq N$, we are interested in the $s$-th moment of the exponential sum
\[R_{k,q}(\alpha;K) = R_k(\alpha) = \sum_{\deg(f-F)<K} \mu_{k,q}(f)\psi(f\alpha).\]
We are able to extend Theorem \ref{L1-thm-ff} to estimates for all $k\geq 2$ and $s>0$.

\begin{theorem}\label{vrho-thm}
    Suppose $s>0$ and let $k \geq 2$ be an integer. Let $F \in \F_q[t]$ have degree $N$, and let $K\leq N$. Then for any $\epsilon>0$, one has that
    \<\label{R_k-asymp-eq}
    \int_\T |R_k(\alpha)|^s \d\alpha \asymp_{k,q,s,\epsilon} \begin{cases}
        \hK^{\frac{s}{k+1}} \quad &\text{if } s<1+\frac{1}{k}, \\
        \hK^{\frac{1}{k}}\log(\hK) \quad &\text{if } s=1+\frac{1}{k}, \\
        \hK^{s-1} \quad &\text{if } s>1+\frac{1}{k}
    \end{cases}
    \>
    as $N \rightarrow \infty$ whenever $K > (\vrho_{k,s}+\epsilon) N$, where
    \<\label{vrho_k,s-def-eq}
    \vrho_{k,s} = \begin{cases}
        \frac{1}{2} \quad &\text{if } s\leq 1+\frac{1}{k}, \\
        \frac{1}{2(k+1)}(1+\frac{1}{s-1}) \quad &\text{if } 1+\frac{1}{k} < s < 2, \\
        \frac{1}{k+1} \quad &\text{if } s\geq 2.
    \end{cases}
    \>
\end{theorem}

This is analogous to the author's main result over $\Z$ in \cite{doyle:2026+}, with the notable exception that we are able to prove here a precise estimate in the critical case $s=1+\frac{1}{k}$. Unfortunately, the constraint on $K$ in Theorem \ref{vrho-thm} is slightly more restrictive than its integer analogue; to surpass the barrier $K=N/2$ as in \cite{doyle:2026+} would require the development of a robust $\F_q[t]$-analogue of van der Corput's method in one and two variables.

The early work of Balog and Ruzsa \cite{balog-ruzsa:2001} examining $S_k(\alpha;N)$ was partially motivated by a lower bound for the $L^1$-mean of the M\"obius-twisted exponential sum. In particular, applying their arguments to short intervals, the author \cite{doyle:2026+} obtained the lower bound
\[\int_0^1 \bigg|\sum_{N-K<n \leq N} \mu(n)e(n\alpha)\bigg| \d\alpha \gg K^{\frac{1}{6}}\]
whenever $K \gg N^{0.49685}$, improving the condition $K \gg N^{\frac{9}{17}+\epsilon}$ given by Sun \cite{sun:2023}. An identical argument may be used in the function field setting.

\begin{theorem}
    Let $F \in \F_q[t]$ be of degree $N$ and fix $\epsilon>0$. If $K > (\frac{1}{2}+\epsilon)N$, then one has the bound
    \<\label{Mobius-thm-eq}
    \int_\T \bigg|\sum_{|f-F|<\hK} \mu_q(f)\psi(f\alpha)\bigg| \d\alpha \gg_{q,\epsilon} \hK^{\frac{1}{6}}.
    \>
\end{theorem}
\begin{proof}
    The proof proceeds exactly as in the integer setting; see \cite{balog-ruzsa:2001} for the original argument. Write
    \[G(\alpha) = \sum_{|f-F|<\hK} \mu_q(f)\psi(f\alpha).\]
    It follows from orthogonality that
    \[\int_\T \overline{G(\beta)}G(\alpha+\beta)\d\beta = \sum_{|f-F|<\hK} \mu_q^2(f)\psi(f\alpha) = R_2(\alpha),\]
    and so Theorem \ref{vrho-thm} gives that
    \[\hK^{\frac{1}{3}} \ll \int_\T |R_2(\alpha)| \d\alpha \leq \int_\T \int_\T |G(\beta)G(\alpha+\beta)|\d\beta\d\alpha = \bigg(\int_\T |G(\alpha)|\d\alpha\bigg)^2. \qedhere\]
\end{proof}

In \cite{doyle:2026+}, the admissible interval length $N^{0.49685}$ is shown to be entirely dependent on a middle part estimate of the shape
\[\sum_{N-K < n \leq N} \bigg|\sum_{\substack{y < d \leq z \\ d^k|n}} \mu(d) \bigg|^2 \ll Ky^{1-k} + N^\epsilon K^{\frac{1}{k}} + N^{\delta_k+\epsilon} \qquad (1 \leq y < z).\]
Indeed, one need only to improve the exponent $\delta_k$ to obtain the result for shorter intervals in $\Z$. In $\F_q[t]$, there is an additional obstruction to improvements on \eqref{vrho_k,s-def-eq} when the characteristic $p$ divides $k$. For $a$ and $b$ positive integers and $c$ a non-negative integer, we say $a^c\|b$ if $c$ is the largest integer such that $a^c|b$. Our results follow from an analogous middle part estimate. Define
\<\label{c_i-def-eq}
c_i(f) = \sum_{\substack{d \in \calM_i \\ d^k|f}} \mu(d).
\>
\begin{lemma}\label{c_i-bd-lemma}
    Suppose $p^c \| k$. Then one has that
    \<\label{c_i-bd-lemma-eq}
    \sum_{|f-F|<\hK} |c_i(f)|^2 \ll \hK\hi^{1-k} + \hN^{\frac{1}{k}-\frac{2}{p^ck}+\epsilon}\hK^{\frac{2}{p^ck}} + \hN^{\frac{1}{k+1}+\epsilon}.
    \>
\end{lemma}

The middle term in \eqref{c_i-bd-lemma-eq} is another barrier to taking $K$ much smaller than $N/2$, even if we were to improve the term $\hN^{\frac{1}{k+1}+\epsilon}$. Luckily, this does not obstruct our approach to Theorem \ref{vrho-thm}.

In \S2, we begin with a more detailed discussion of the function field setting and the construction of our exponential sum. In \S3, we introduce an array of different moment estimates for the relating to the middle parts $c_i(f)$, many of which depend on Lemma \ref{c_i-bd-lemma}. In \S4, we use these moment estimates to prove Theorem \ref{vrho-thm}.  Finally, in \S5 we prove Lemma \ref{c_i-bd-lemma}. This section makes use of our heaviest machinery, notably the Riemann hypothesis for $L$-functions over short interval characters. We note that the techniques used in \S\S3-4 are identical to those used in the integer setting of \cite{doyle:2026+}; the novelty here is the in the proof of Lemma \ref{c_i-bd-lemma}.

We use $\ll$ and $\gg$ throughout to denote Vinogradov's notation, in which $f \ll g$ if there exists a constant $C$ such that $f(N) \leq Cg(N)$ whenever $N$ is sufficiently large. Equivalently, we will also use when convenient Landau's $O$ notation, writing $f = O(g)$ when $f \ll g$. If $f \ll g$ and $g \ll f$ then we write $f \asymp g$. Throughout our arguments, the implicit constants are allowed to depend on $k$, $q$, $s$, and $\epsilon$, unless otherwise specified. Finally, we remind the reader of the unconventional notation $\hA = q^A$.

\section{Preliminary Discussions}

We are concerned with the ring of polynomials $\F_q[t]$ over the finite field $\F_q$, where $q=p^\ell$ for some prime $p$. Denote by $\calM$ the set of monic polynomials in $\F_q[t]$, with $\calM_n$ denoting those monic polynomials of degree $n$. We also regularly use the notation $\calM_{<n}$ or $\calM_{>n}$ for their obvious meanings. Associated to $\F_q[t]$ is its field of fractions $\K = \F_q(t)$, which may be completed at $\infty$ to obtain the field of formal Laurent series $\K_\infty = \F_q((1/t))$. An arbitrary $\alpha \in \K_\infty$ is of the shape $\alpha = \sum_{-\infty < i \leq n} a_it^i$ for some $n \in \Z$ and coefficients $a_i = a_i(\alpha) \in \F_q$, and for a given $\alpha$ of this form we write $\ord(\alpha) = n$ and $|\alpha| = q^{\ord(\alpha)}$. By convention, we define $\ord(0) = -\infty$ and $|0| = 0$.

Naturally one may then consider the ``unit interval'' given by $\T = \{\alpha \in \K_\infty: |\alpha|<1\}$, and it is easy to see that one may decompose each $\alpha \in \K_\infty$ uniquely as $\alpha = \lfloor \alpha \rfloor + \{\alpha\}$, where $\{\alpha\} \in \T$ and $\lfloor \alpha\rfloor \in \F_q[t]$. We define $\|\alpha\| = |\{\alpha\}|$. Finally, we define a Haar measure $\d\alpha$ on $\K_\infty$, normalized so that $\int_\T \d\alpha = 1$.

From here we are prepared to define an additive character on $\K_\infty$ which plays the role of the exponential in our sum. As the characteristic of $\F_q$ is $p$, we have a non-trivial additive character $e_q:\F_q \rightarrow \C^\times$ which is given by $e_q(a) = \exp(\frac{2\pi i \tr(a)}{p})$, with $\tr:\F_q\rightarrow \F_p$ denoting the familiar trace map. We may then induce an additive character on $\K_\infty$ by assigning 
\<\label{psi-def-eq}
\psi(\alpha) := e_q(a_{-1}(\alpha)).
\>

It is clear by definition that $\psi$ is periodic modulo $\T$, and furthermore $\psi$ obeys the orthogonality identity
\<\label{ortho-ident-eq}
\int_\T \psi(f\alpha)\d\alpha = \begin{cases}
    1 \quad &\text{if } f=0, \\
    0 \quad &\text{if } f \in \F_q[t]\setmin \{0\},
\end{cases}
\>
whence an analogue of Parseval's identity immediately follows: for any sequence $(c_f)_{f \in \F_q[t]} \in \C$, one has
\<\label{Parseval-eq}
\int_\T \bigg|\sum_{f \in \F_q[t]} c_f \psi(f\alpha)\bigg|^2 \d\alpha = \sum_{f \in \F_q[t]} |c_f|^2
\>
Furthermore, one may see that a discrete orthogonality identity of the form
\<\label{disc-ortho-ident-eq}
\sum_{|f| < \hM} \psi(f\alpha) = \begin{cases}
    \hM \quad &\text{if } |\alpha| < \hM^{-1}, \\
    0 \quad &\text{else}
\end{cases} 
\>
holds. This is a notable difference from the integer setting, permitting logarithmic savings in many parts of our arguments. Finally, we observe a convenient translation-dilation invariance property.

\begin{lemma}\label{trans-dil-invar}
    Let $\calS$ be a finite subset of $\F_q[t]$, and for $u \in \F_q[t]$ and $v \in \K_\infty$, denote by $u\calS+v$ the set $\{uf+v:f \in \calS\}$. Then for $s>0$, one has that
    \[\int_\T \bigg|\sum_{f \in \calS} \psi(f\alpha)\bigg|^s \d\alpha = \int_\T \bigg|\sum_{f \in u\calS+v} \psi(f\alpha)\bigg|^s \d\alpha.\]
\end{lemma}
\begin{proof}
    This follows simply from observing that the Haar measure $\d\alpha$ permits a change of variables under the rule $\d(h\alpha) = |h|\d\alpha$. Thus we have that
    \begin{align*}
        \int_\T \bigg|\sum_{f \in u\calS+v} \psi(f\alpha)\bigg|^s \d\alpha &= \int_\T \bigg|\sum_{f \in \calS} \psi((uf+v)\alpha)\bigg|^s \d\alpha \\
        &= \int_\T \bigg|\sum_{f \in \calS} \psi(f\cdot(u\alpha))\bigg|^s \d\alpha \\
        &= |u|^{-1}\int_{u\T} \bigg|\sum_{f \in \calS} \psi(f\alpha)\bigg|^s \d\alpha.
    \end{align*}
    The integrand is $\T$-periodic, and so we gain a factor of $|u|$, which confirms the result.
\end{proof}

\section{Decomposition Estimates for $R_k(\alpha)$}

The typical approach to dealing with the exponential sum $R_k(\alpha)$ revolves around the middle parts $c_i(f)$ defined in \eqref{c_i-def-eq}. We decompose $\mu_{k,q}$ as
\[\mu_{k,q}(f) = \sum_{\substack{d \in \calM \\ d^k|f}} \mu_q(d) = \sum_{i} \sum_{\substack{d \in \calM_i \\ d^k|f}} \mu_q(d) = \sum_i c_i(f),\]
and so we can define
\<\label{T_i-def-eq}
    T_i(\alpha) = \sum_{|f-F|<\hK} c_i(f)\psi(f\alpha)
\>
so that $R_k(\alpha) = \sum_i T_i(\alpha)$. It is convenient also to have the top and bottom segments
\<\label{h_i,H_i-def-eq}
h_i(\alpha) = \sum_{j<i} T_j(\alpha), \quad H_i(\alpha) = \sum_{j \geq i} T_j(\alpha)
\>
so that $R_k(\alpha) = h_i(\alpha) + H_i(\alpha)$ for any $i$. Each of the moment estimates of Theorem \ref{vrho-thm} are constructed using various estimates on these components. Throughout the remainder of this section we assume Lemma \ref{c_i-bd-lemma}, which is proven in \S5.

\begin{lemma}\label{T_i-h_i-L1-bd-lemma}
    One has that
    \[\int_\T |T_i(\alpha)|\d\alpha \ll \hi \qquad \text{and} \qquad \int_\T |h_i(\alpha)|\d\alpha \ll \hi.\]
\end{lemma}
\begin{proof}
    The second bound follows from the first simply by applying the triangle inequality. Now observe that, upon applying the definition \eqref{T_i-def-eq} of $T_i(\alpha)$ and the definition \eqref{c_i-def-eq} of $c_i(f)$, swapping the order of summation and applying the triangle inequality gives the bound
    \[\int_\T |T_i(\alpha)| \d\alpha \leq \sum_{d \in \calM_i} \int_\T \bigg|\sum_{|fd^k-F|<\hK} \psi(fd^k\alpha)\bigg|\d\alpha.\]
    If $\hi \geq \hK^{\frac{1}{k}}$, then the exponential sum has at most one summand, and we have that
    \[\int_\T |T_i(\alpha)| \d\alpha \leq \sum_{d \in \calM_i} 1 = \hi.\]
    
    If $\hi < \hK^{\frac{1}{k}}$, the indexing set is $d^k\calS+F$, where $\calS = \{|f|<\hK\hi^{-k}\}$. We may therefore apply Lemma \ref{trans-dil-invar} and \eqref{disc-ortho-ident-eq} to see that this integral evaluates to 1, and so again we have the desired bound for the first integral.
\end{proof}

\begin{lemma}\label{H_i-pw-lemma}
    Suppose $p^c \| k$. For any $\alpha \in \T$ and any $\epsilon > 0$, one has that
    \[|H_i(\alpha)| \ll \hK\hi^{\frac{1-k}{2}} + \hN^{\frac{1}{2k}-\frac{1}{p^ck}+\epsilon}\hK^{\frac{1}{2}+\frac{1}{p^ck}} + \hK^{\frac{1}{2}}\hN^{\frac{1}{2(k+1)}+\epsilon}.\]
\end{lemma}
\begin{proof}
    This is just an application of the Cauchy-Schwarz inequality and Lemma \ref{c_i-bd-lemma}. We have that
    \[H_i(\alpha) \leq \hK^{\frac{1}{2}}\bigg(\sum_{f \in \calM_{\geq i}} |c_i(f)|^2\bigg)^{\frac{1}{2}} \ll \hK^{\frac{1}{2}} (\hK\hi^{1-k} + \hN^{\frac{1}{k}-\frac{2}{p^ck}+\epsilon}\hK^{\frac{2}{p^ck}} + \hN^{\frac{1}{k+1}+\epsilon})^{\frac{1}{2}}. \qedhere\]
\end{proof}

\begin{lemma}\label{T_i,H_i-L2-bd-lemma}
    Suppose $p^c \| k$. Then for any $\epsilon>0$, one has that
    \[\int_\T |T_i(\alpha)|^2 \d\alpha \ll K\hi^{1-k} + \hN^{\frac{1}{k}-\frac{2}{p^ck}+\epsilon}\hK^{\frac{2}{p^ck}} + \hN^{\frac{1}{k+1}+\epsilon}\]
    and
    \[\int_\T |H_i(\alpha)|^2 \d\alpha \ll K\hi^{1-k} + \hN^{\frac{1}{k}-\frac{2}{p^ck}+\epsilon}\hK^{\frac{2}{p^ck}} + \hN^{\frac{1}{k+1}+\epsilon}.\] 
\end{lemma}
\begin{proof}
    This follows immediately from Parseval's identity \eqref{Parseval-eq} and Lemma \ref{c_i-bd-lemma}.
\end{proof}

\begin{corollary}\label{T_i-Ls-bd-corollary}
    Suppose $p^c \| k$. For $1 \leq s \leq 2$ and $\epsilon>0$, one has that
    \[\int_\T |T_i(\alpha)|^s \d\alpha \ll \hK^{s-1}\hi^{1-k(s-1)} + (\hN^{\frac{1}{k}-\frac{2}{p^ck}+\epsilon}\hK^{\frac{2}{p^ck}})^{s-1}\hi^{2-s} + \hN^{\frac{s-1}{k+1}+\epsilon}\hi^{2-s}.\]
\end{corollary}
\begin{proof}
    This follows immediately upon using H\"older's inequality to interpolate the bounds of Lemmata \ref{T_i-h_i-L1-bd-lemma} and \ref{T_i,H_i-L2-bd-lemma}.
\end{proof}
\begin{corollary}\label{h_i-Ls-bd-corollary}
    Suppose $p^c \| k$. For $1 \leq s < 1+\frac{1}{k}$ and any $\epsilon>0$, one has that
    \[\int_\T |h_i(\alpha)|^s\d\alpha \ll \hK^{s-1}\hi^{1-k(s-1)} + (\hN^{\frac{1}{k}-\frac{2}{p^ck}+\epsilon}\hK^{\frac{2}{p^ck}})^{s-1}\hi^{2-s} + \hN^{\frac{s-1}{k+1}+\epsilon}\hi^{2-s}.\]
\end{corollary}
\begin{proof}
    The case $s=1$ is given by Lemma \ref{T_i-h_i-L1-bd-lemma}. For $s>1$, we use a polynomial weight to save a logarithmic factor, writing
    \begin{align*}
        \int_\T |h_i(\alpha)|^s \d\alpha &= \int_\T \bigg|\sum_{j<i} (i-j)^{-1}(i-j)T_j(\alpha)\bigg|^s\d\alpha \\
        &\ll \bigg(\sum_{j<i} (i-j)^{-\frac{s}{s-1}}\bigg)^{s-1} \sum_{j<i}(i-j)^s\int_\T |T_j(\alpha)|^s\d\alpha.
    \end{align*}
    The first sum is $O(1)$ since $s>1$. We apply Corollary \ref{T_i-Ls-bd-corollary} to the remaining integral, observing that since $s<1+\frac{1}{k}$ we have that $1-k(s-1) > 0$ and so
    \begin{align*}
        \sum_{j<i}(i-j)^s\int_\T |T_j(\alpha)|^s\d\alpha &\ll \hK^{s-1}\hi^{1-k(s-1)} + (\hN^{\frac{1}{k}-\frac{2}{p^ck}+\epsilon}\hK^{\frac{2}{p^ck}})^{s-1}\hi^{2-s} + \hN^{\frac{s-1}{k+1}+\epsilon}\hi^{2-s}. \qedhere
    \end{align*}
\end{proof}

It is useful to observe that the bounds of Corollaries \ref{T_i-Ls-bd-corollary} and \ref{h_i-Ls-bd-corollary} are dominated by the term $\hK^{s-1}\hi^{1-k(s-1)}$ whenever $K > (\frac{1}{2}+\epsilon)N$ and $\hi \ll \hK^{\frac{1}{k+1}}$.

\begin{lemma}\label{hH-mixed-moment-lemma}
    Let $D = (\frac{1}{k+1})K$ and $\epsilon>0$. For $K > (\frac{1}{k+1}+\epsilon)N$, one has that
    \[\int_\T |h_D(\alpha)H_D(\alpha)| \d\alpha = o(\hK).\]
\end{lemma}
\begin{proof}
    It is sufficient to prove the result for $(\frac{1}{k+1}+\epsilon)N < K < (\frac{1}{k+1}+\delta)N$, where $\delta>\epsilon$ is suitably small. Set $C = (\frac{1}{100})\epsilon N$, say, and write
    \<\label{hH-mixed-moment-decomp-eq-1}
    \int_\T |h_D(\alpha)H_D(\alpha)| \d\alpha \ll \int_\T |h_C(\alpha)H_D(\alpha)| \d\alpha + \int_\T \bigg|\sum_{C\leq i<D} T_i(\alpha)H_D(\alpha)\bigg|\d\alpha.
    \>
    Now, we have
    \[\int_\T |h_C(\alpha)H_D(\alpha)| \d\alpha \ll \bigg(\sup_{\alpha \in \T} |H_D(\alpha)|\bigg) \int_\T |h_C(\alpha)|\d\alpha\]
    whence the bounds of Lemmata \ref{T_i-h_i-L1-bd-lemma} and \ref{H_i-pw-lemma} provide the estimate
    \[\int_\T |h_C(\alpha)H_D(\alpha)|\d\alpha \ll (\hK^{\frac{k+3}{2(k+1)}} + \hK^{\frac{1}{2}}\hN^{\frac{1}{2(k+1)}+\epsilon'})\hC\log(\hN)\]
    for any $\epsilon'>0$. If one sets $\epsilon' = \frac{1}{100}\delta$, say, then each term in the parentheses is $o(\hK^{1-2\epsilon'})$, so that
    \<\label{hH-mixed-eq-1}
    \int_0^1 |h_C(\alpha)H_D(\alpha)|\d\alpha \ll (\hK^{1-2\epsilon'})\hC \log N = o(\hK).
    \>
    
    For the remaining integral in \eqref{hH-mixed-moment-decomp-eq-1}, the triangle inequality gives
    \[\int_0^1 \bigg|\sum_{C\leq i<D} T_i(\alpha)H_D(\alpha)\bigg|\d\alpha \leq \sum_{C<i<D}\int_0^1 |T_i(\alpha)H_D(\alpha)|\d\alpha,\]
    and H\"older's inequality carefully applied to each integral yields the bound
    \begin{multline*}
        \int_0^1 \bigg|\sum_{C\leq i<D} T_i(\alpha)H_D(\alpha)\bigg|\d\alpha \leq \sum_{C\leq i<D} \bigg(\sup_\alpha |H_D(\alpha)|\bigg)^{\frac{1}{i}}\bigg(\int_0^1 |T_i(\alpha)|^2\d\alpha\bigg)^{\frac{i-1}{2i}} \\
        \times \bigg(\int_0^1 |T_i(\alpha)|\d\alpha\bigg)^{\frac{1}{i}}\bigg(\int_0^1|H_D(\alpha)|^2\d\alpha\bigg)^{\frac{i-1}{2i}}.
    \end{multline*}
    Appropriately applying the bounds of Lemmata \ref{T_i,H_i-L2-bd-lemma} and \ref{H_i-pw-lemma} and of Corollary \ref{T_i-Ls-bd-corollary}, we see that, for any $\epsilon''>0$, the above integral is bounded by
    \[(\log \hN) \sum_{C\leq i<D}  (\hK^{\frac{k+3}{2(k+1)}} + \hK^{\frac{1}{2}}\hN^{\frac{1}{2(k+1)}+\epsilon''})^{\frac{1}{i}}(\hK\hi^{1-k} + \hN^{\frac{1}{k+1}+\epsilon''})^{\frac{i-1}{2i}}(\hK^{\frac{2}{k+1}} + \hN^{\frac{1}{k+1}+\epsilon''})^{\frac{i-1}{2i}}\]
    Setting $\epsilon''= \frac{1}{100}\delta$, each term in the product above is at most $O(\hK^{1-\epsilon''})$, and so
    \<\label{hH-mixed-eq-2}
    \int_0^1 \bigg|\sum_{C\leq i<D} T_i(\alpha)H_D(\alpha)\bigg|\d\alpha \ll \hK^{1-\epsilon''}\log \hN \sum_{C\leq i<D} \hi^{\frac{1}{i}} \ll \hK^{1-\epsilon''}(\log \hN)^2 = o(\hK).
    \>
    Applying \eqref{hH-mixed-eq-1} and \eqref{hH-mixed-eq-2} to \eqref{hH-mixed-moment-decomp-eq-1}, we have the result.
\end{proof}

\section{The Proof of Theorem \ref{vrho-thm}}

We prove Theorem \ref{vrho-thm} through a number of cases, considering each of the upper and lower bounds separately. Lemma \ref{subcrit-lower-lemma} accounts for an error in a previous argument of the author (see Lemma 3.2 of \cite{doyle:2026}). That argument claims that, if $0\leq |\ell|<|u|^k \leq \hD$ and $(\ell,u^k)$ is $k$-free, then for any $\alpha$ with $|u^k\alpha-\ell|<\hD\hN^{-1}$ we have the bound $|h_D(\alpha)| > \hN|d|^{-k}$. This is not true. Instead, we must restrict to a much smaller (but still sufficiently large) subset $\calY \subset \T$ on which this bound holds. The argument of Lemma \ref{calY-lemma} is sufficient for this purpose, though one must make slight adjustments to apply this argument to the exponential sum considered in \cite{doyle:2026}.

Throughout all of these arguments, we again let $D = (\frac{1}{k+1})K$.

\subsection{The Case $s<1+\frac{1}{k}$}

The lower bound in the subcritical case is the most involved of our arguments. As mentioned, we must construct a set $\calY$ on which the $h_D(\alpha)$ component is dominant. Using the orthogonality identity \eqref{disc-ortho-ident-eq}, we may write
\begin{align*}
    h_D(\alpha) &= \sum_{d \in \calM_{< D}} \mu_q(d) \sum_{\substack{|f-F|<\hK \\ d^k|f}} \psi(f\alpha) \\
    &= \sum_{d \in \calM_{<D}} \dfrac{\mu_q(d)}{|d|^k} \sum_{|a|<|d|^k} \sum_{|f-F|<\hK} \psi\bigg(f\bigg(\alpha-\frac{a}{d^k}\bigg)\bigg).
\end{align*}
We may then sort according to the greatest common divisor $(a,d^k)$, using an inclusion-exclusion principle to write
\[h_D(\alpha) = \sum_{d \in \calM_{<D}} \mu_q(d)\Bigg(\sum_{\substack{m \in \calM_{<D-\deg d} \\ (d,m)=1}}\dfrac{\mu_q(m)}{|m|^k}\Bigg)\Bigg(\frac{1}{|d|^k}\sum_{\substack{|a|<|d|^k \\ \mu_{k,q}((a,d^k))=1}} \sum_{|f-F|<\hK} \psi\bigg(f\bigg(\alpha-\frac{a}{d^k}\bigg)\bigg)\Bigg),\]
which we abbreviate as
\<\label{L1-lower-h_D-decomp-eq-1}
h_D(\alpha) = \sum_{d \in \calM_{<D}} \mu_q(d)b_dG_d(\alpha).
\>
    
The term $b_d$ is straightforward. We have that
\<\label{subcrit-lower-zeta-bd-eq}
b_d = \sum_{\substack{m \in \calM_{\leq D-\deg(d)} \\ (m,d)=1}} \dfrac{\mu_q(m)}{|m|^k} \geq 1 - \sum_{\substack{m \in \calM_{>0} \\ (m,d) = 1}}\dfrac{1}{|m|^k} = 2 - \zeta_q(k)\prod_{\pi|d}(1-|\pi|^{-k}),
\>
where $\zeta_q(s)$ is the zeta function for $\F_q[t]$ given by
\[\zeta_q(s) = \sum_{f \in \calM} |f|^{-s} = \frac{1}{1-q^{1-s}}.\] 
It is easy to see from the definition that if $d \neq 1$, then the expression on the right hand side of \eqref{subcrit-lower-zeta-bd-eq} is bounded below by $\frac{1}{2}$. If $d=1$, then clearly one has that 
\[b_d = \sum_{\substack{m \in \calM_{\leq D-\deg(d)} \\ (m,d)=1}} \dfrac{\mu_q(m)}{|m|^k} \rightarrow \zeta_q(k)^{-1} \geq \frac{1}{2},\]
and so in any case for $N$ large enough this will be bounded below by, $1/3$, say. On the other hand, we also have that $b_d \leq \zeta_q(k) \ll_k 1$, and so $b_d$ is effectively constant. 

By orthogonality, we have that 
\<\label{G_d-in-X_d-eq}
G_d(\alpha) = \begin{cases}
    \psi(F(\alpha-\frac{a}{d^k}))\hK|d|^{-k} \quad &\text{if } |\alpha - \frac{\ell}{d^k}|<\hK^{-1}, \\
    0 \quad &\text{else.}
\end{cases}
\>
Because of this, and because $b_d$ is roughly constant, we may hope that for $\alpha$ close to $\frac{\ell}{d^k}$ for $d$ squarefree and small, the term $\mu_q(d)b_dG_d(\alpha)$ is the dominant term in \eqref{L1-lower-h_D-decomp-eq-1}. Na\"ively, then, we define for $d \in \calM_{<D}$ squarefree the set
\[\calX_d = \bigcup_{\substack{|\ell|<|d|^k \\ (\ell,d^k) \; k\text{-free}}}\bigg\{\alpha \in \T: \bigg|\alpha - \frac{\ell}{d^k}\bigg|<\hK^{-1}\bigg\}.\]
Unfortunately the sets $\calX_d$ are not generally disjoint, and so we risk having large contributions from multiple $d$. We pass instead to large subsets of the $\calX_d$ which are disjoint.

\begin{lemma}\label{calY-lemma}
    Let $\eta > 0$ be sufficiently small, and set $D_0 = \eta D$. For $|d| \in \calM_{<D_0}$, there exists $\calY_d \subseteq \calX_d$ such that $|\calY_d| \gg |\calX_d|$ and $|h_D(\alpha)| \gg \hK|d|^{-k}$ for $\alpha \in \calY_d$. Furthermore, whenever $d,d' \in \calM_{<D_0}$ and $d \neq d'$, the sets $\calY_d$ and $\calY_{d'}$ are disjoint.
\end{lemma}

\begin{proof}
    Define $\calY_d$ to be the subset of $\calX_d$ containing all $\alpha$ such that 
    \<\label{Y_d-def-eq}
    \calY_d = \bigg\{\alpha \in \calX_d:\;\sum_{\substack{d' \in \calM_{<D} \\ d' \neq d}} |G_{d'}(\alpha)| \leq \frac{1}{10}\hK|d|^{-k}\bigg\}.
    \>
    and observe that \eqref{G_d-in-X_d-eq} and \eqref{Y_d-def-eq} together imply that the $\calY_d$ are pairwise disjoint. It is easy to see from the definitions \eqref{G_d-in-X_d-eq} and $\eqref{Y_d-def-eq}$ with \eqref{L1-lower-h_D-decomp-eq-1} that $|h_D(\alpha)| \gg \hK|d|^{-k}$ when $\alpha \in \calY_d$.

    The main effort then is to show that the sets $\calY_d$ are sufficiently large. For this we follow the arguments of Balog and Ruzsa \cite{balog-ruzsa:2001}. Let $\calZ_d = \calX_d\setmin\calY_d$, and observe that \eqref{G_d-in-X_d-eq} and \eqref{Y_d-def-eq} imply that
    \<\label{quasi-orthog-eq-1}
    \sum_{\substack{d' \in \calM_{<D} \\ d' \neq d}} \int_{\calZ_d} |G_d(\alpha)G_{d'}(\alpha)|\d\alpha = \int_{\calZ_d} |G_d(\alpha)|\sum_{\substack{d' \in \calM_{<D} \\ d' \neq d}} |G_{d'}(\alpha)| \d\alpha \geq |\calZ_d| \frac{1}{10}\hK^2 |d|^{-2k}.
    \>
    But by the definition of $G_d(\alpha)$, we have that
    \begin{align*}
        \int_{\calZ_d} |G_d(\alpha)G_{d'}(\alpha)|\d\alpha &\leq |dd'|^{-k} \sum_{\substack{|a|<|d|^k \\ \mu_{q,k}((a,d^k))=1}}\sum_{\substack{|a'|<|d'|^k \\ \mu_{q,k}((a',(d')^k))=1}} \int_\T \bigg|\nu\bigg(\alpha-\frac{a}{d^k}\bigg)\nu\bigg(\alpha-\frac{a'}{(d')^k}\bigg)\bigg|\d\alpha,
    \end{align*}
    where
    \[\nu(\alpha) = \sum_{|f|<\hK} \psi(f\alpha).\]
    It follows from the orthogonality relation \eqref{disc-ortho-ident-eq} that the integral satisfies
    \[\int_\T\bigg|\nu\bigg(\alpha-\frac{a}{d^k}\bigg)\nu\bigg(\alpha-\frac{a'}{(d')^k}\bigg)\bigg|\d\alpha = \begin{cases}
        \hK \quad &\text{if } |\frac{a}{d^k}-\frac{a'}{(d')^k}|<\hK^{-1}, \\
        0 \quad &\text{else,}
    \end{cases}\]
    and so we see that 
    \begin{multline*}
        \int_{\calZ_d} |G_d(\alpha)G_{d'}(\alpha)|\d\alpha \\
        \leq \frac{\hK}{|dd'|^{k}} \#\bigg\{|a|<|d|^k,|a'|<|d'|^k : \bigg|\dfrac{a}{d^k} - \dfrac{a'}{(d')^k}\bigg| < \hK^{-1}\; \text{and}\;(a,d^k),(a',(d')^k) \;k\text{-free}\bigg\}.
    \end{multline*}

    This is equivalent to counting the number of solutions to
    \<\label{L1-lower-cong-eq-1}
    a\dfrac{(d')^k}{(d,d')^k}-a'\dfrac{d^k}{(d,d')^k} \equiv m \mmod{[d,d']^k},
    \>
    where $|a|<|d|^k$, $|a'|<|d'|^k$ and $(a,d^k)$, $(a',(d')^k)$ are $k$-free, and $0<|m|<|[d,d']|^k\hK^{-1}$. Any solution to this congruence also satisfies
    \<\label{L1-lower-cong-eq-2}
    a\frac{(d')^k}{(d,d')^k} \equiv m \,\,\bigg({\rm{mod}}\,\,\frac{d^k}{(d,d')^k}\bigg).
    \>
    There are exactly $|(d,d')|^k$ choices of $a$ satisfying \eqref{L1-lower-cong-eq-2} for each particular choice of $m$, and returning to \eqref{L1-lower-cong-eq-1} there is a unique $a'$ for each choice of $a$. Thus we have that
    \[\int_{\calZ_d} |G_d(\alpha)G_{d'}(\alpha)| \d\alpha \leq \frac{\hK}{|dd'|^k} |(d,d')|^k \bigg(\dfrac{|[d,d']|^k}{\hK}\bigg) = 1,\]
    and so \eqref{quasi-orthog-eq-1} gives that
    \[|\calZ_d| \ll \frac{\hD |d|^{2k}}{\hK^2}.\]
    Thus for $|d|<\hD_0 = \eta \hD$ for some sufficiently small $\eta>0$, we see that $|\calZ_d| = o(|d|^k\hK^{-1})$, and we are done.
\end{proof}

Having such a collection of disjoint sets at the ready, we may proceed to the proof of the lower bound when $s<1+\frac{1}{k}$.

\begin{lemma}\label{subcrit-lower-lemma}
    Let $\epsilon>0$. For $K > (\frac{1}{2}+\epsilon)N$, we have that
    \[\int_\T |R_k(\alpha)|^s \d\alpha \gg \hK^{\frac{s}{k+1}}.\]
\end{lemma}
\begin{proof}
    It follows from H\"older's inequality and upper bound given in the next lemma to prove it for $s=1$. To see this, suppose the result holds for $s=1$, and observe that
    \[\int_\T |R_k(\alpha)|^s \d\alpha \gg \bigg(\int_\T |R_k(\alpha)| \d\alpha\bigg)^s \gg \hK^{\frac{s}{k+1}}\]
    when $s>1$, and for $s < 1$ we may apply H\"older's inequality to obtain
    \[\hK^{\frac{1}{k+1}} \ll \int_\T |R_k(\alpha)| \d\alpha \ll \bigg(\int_\T |R_k(\alpha)|^s \d\alpha\bigg)^{\frac{\delta}{1+\delta-s}}\bigg(\int_\T |R_k(\alpha)|^{1+\delta} \d\alpha\bigg)^{\frac{1-s}{1+\delta-s}},\]
    for any $0<\delta<\frac{1}{k}$, which then implies the result by the upper bound for the $(1+\delta)$-th moment.

    Thus we are left with the case $s=1$. For $\eta>0$ sufficiently small and $D_0 = \eta D$, let
    \[\calY = \bigcup_{\substack{d \in \calM_{<D_0} \\ \mu_q^2(d)=1}} \calY_d,\]
    where $\calY_d$ is defined as in \eqref{Y_d-def-eq}. By the previous lemma, the $\calY_d$ are disjoint and obey the bound $|\calY_d| \asymp |d|^k\hK^{-1}$, so that 
    \[|\calY| \asymp \sum_{d \in \calM_{<D_0}} \frac{|d|^k}{\hK} \asymp \eta^{k+1}.\]
    Thus if we write
    \<\label{L1-Lower-decomp-eq-1}
    \int_\T |R_k(\alpha)| \d\alpha \geq \int_{\calY} |h_D(\alpha)| \d\alpha - \int_{\calY} |H_D(\alpha)| \d\alpha,
    \>
    an application of the Cauchy-Schwarz inequality, Parseval's identity and Lemma \ref{c_i-bd-lemma} quickly reveals that the second term may be bounded as
    \<\label{L1-lower-H_D-bd-eq-1}
        \int_{\calY} |H_D(\alpha)| \d\alpha \leq |\calY|^{\frac{1}{2}} \bigg(\sum_{i\geq D} \sum_{|f-F|<\hK}|c_i(f)|^2\bigg)^{\frac{1}{2}} \ll \eta^{\frac{k+1}{2}}\hK^{\frac{1}{k+1}}
    \>
    whenever $K > (\frac{1}{2}+\epsilon)N$.

    For the first term in \eqref{L1-Lower-decomp-eq-1}, the previous lemma indicates that
    \<\label{L1-lower-h_D-bd-eq-2}
    \int_\calY |h_D(\alpha)| \d\alpha \gg \sum_{d \in \calM_{<D_0}} \frac{\mu_q^2(d)}{|d|^k}|\calY_d|\hK \gg \hD_0 = \eta \hK^{\frac{1}{k+1}}.
    \>
    Upon returning to \eqref{L1-Lower-decomp-eq-1} then, we see by \eqref{L1-lower-H_D-bd-eq-1} and \eqref{L1-lower-h_D-bd-eq-2} that
    \begin{align*}
        \int_\T |R_k(\alpha)|\d\alpha &\gg  (\eta - \eta^{\frac{k+1}{2}})\hK^{\frac{1}{k+1}},
    \end{align*}
    and choosing $\eta>0$ sufficiently small provides the result.
\end{proof}

\begin{lemma}\label{subcrit-upper-lemma}
    Let $\epsilon>0$. For $s<1+\frac{1}{k}$ and $K > (\frac{1}{2}+\epsilon)N$, we have that
    \[\int_\T |R_k(\alpha)|^s\d\alpha \ll \hK^{\frac{s}{k+1}}.\]
\end{lemma}
\begin{proof}
    This follows fairly easily from our moment estimates in the previous section. We immediately have that
    \<\label{subcrit-upper-decomp-eq-1}
    \int_\T |R_k(\alpha)|^s \d\alpha \ll \int_\T |h_D(\alpha)|^s \d\alpha + \int_\T |H_D(\alpha)|^s \d\alpha.
    \>
    By H\"older's inequality and Lemma \ref{T_i,H_i-L2-bd-lemma}, one has the bound
    \<\label{subcrit-upper-H_D-bd-eq}
    \int_\T |H_D(\alpha)|^s \d\alpha \ll \bigg(\int_\T |H_D(\alpha)|^2\d\alpha\bigg)^{\frac{s}{2}} \ll \hK^{\frac{s}{k+1}} + \hN^{\frac{s}{2(k+1)} + \epsilon}.
    \>
    On the other hand, if $1 \leq s < 1+\frac{1}{k}$, then Corollary \ref{h_i-Ls-bd-corollary} shows that the first integral on the right hand side of \eqref{subcrit-upper-decomp-eq-1} obeys
    \[\int_\T |h_D(\alpha)|^s \d\alpha \ll \hK^{s-1}\hD^{1-k(s-1)} + (\hN^{\frac{1}{k}-\frac{2}{p^ck}+\epsilon}\hK^{\frac{2}{p^ck}})^{s-1}\hD^{2-s} + \hN^{\frac{s-1}{k+1}+\epsilon}\hD^{2-s}.\]
    It is easy to verify that if $K > N/2$, then this is $O(\hK^{\frac{s}{k+1}})$. Thus upon applying this to \eqref{subcrit-upper-decomp-eq-1} along with \eqref{subcrit-upper-H_D-bd-eq}, we have the upper bound for $1 \leq s < 1+\frac{1}{k}$. The upper bound for $s<1$ follows trivially from the case $s=1$ by H\"older's inequality.
\end{proof}

\subsection{The Case $s=1+\frac{1}{k}$}

In the critical case we find an advantage in the function field setting which permits a tight upper bound. This arises from the stronger estimate \eqref{disc-ortho-ident-eq}. We first prove the lower bound, noting that previous arguments for this bound in the integer setting (see \S5 of \cite{keil:2013} or Lemma 4.3 of \cite{doyle:2026+}) may be somewhat expedited using Lemma \ref{calY-lemma}.

\begin{lemma}\label{crit-lower-lemma}
    For $K \geq (\frac{1}{2}+\epsilon)N$, one has that
    \[\int_\T |R_k(\alpha)|^{1+\frac{1}{k}} \d\alpha \gg \hK^{\frac{1}{k}}\log(\hK).\]
\end{lemma}
\begin{proof}
    Fix $\epsilon>0$, and suppose $K > (\frac{1}{2}+\epsilon)N$. Observe first that
    \<\label{crit-lower-decomp-eq-1}
    \int_\T |R_k(\alpha)|^{1+\frac{1}{k}} \d\alpha - \int_\T |h_D(\alpha)|^{1+\frac{1}{k}} \d\alpha \ll \int_\T |H_D(\alpha)||h_D(\alpha)|^{\frac{1}{k}}\d\alpha + \int_\T|H_D(\alpha)|^{1+\frac{1}{k}}\d\alpha
    \>
    by the Mean Value Theorem. Applying H\"older's inequality to the first integral on the right, we have that for any $\epsilon'>0$, the bound
    \begin{align*}
        \int_\T|H_D(\alpha)||h_D(\alpha)|^{\frac{1}{k}}\d\alpha &\ll (\hK\hD^{1-k} + \hN^{\epsilon'}\hK^{\frac{1}{k}} + \hN^{\frac{1}{k+1}+\epsilon'})^{\frac{1}{2}}(\hD\log \hK)^{\frac{1}{k}} \\
        &\ll \hK^{\frac{1}{k}}(\log \hK)^{\frac{1}{k}} + \hN^{\frac{1}{2(k+1)} + \frac{\epsilon'}{2}}\hK^{\frac{1}{k(k+1)}}(\log \hK)^{\frac{1}{k}}
    \end{align*}
    holds. In particular, this is $O(\hK^{\frac{1}{k}}(\log\hK)^{\frac{1}{k}})$ when $\epsilon'=\frac{2}{k+1}\epsilon$, say. For the other integral on the right of \eqref{crit-lower-decomp-eq-1}, simply observe that by H\"older's inequality, we have 
    \[\int_\T |H_D(\alpha)|^{1+\frac{1}{k}} \d\alpha \ll \bigg(\int_\T |H_D(\alpha)|^2\bigg)^{\frac{k+1}{2k}} \ll (\hK\hD^{1-k} + \hN^{\epsilon''} \hK^{\frac{1}{k}} + \hN^{\frac{1}{k+1}+\epsilon''})^{\frac{k+1}{2k}} \ll \hK^{\frac{1}{k}},\]
    where $\epsilon'' = (\frac{2k}{k+1})\epsilon$.

    Finally, let $\calY_d$ be defined as in \eqref{Y_d-def-eq}. Recall from Lemma \ref{subcrit-lower-lemma} that whenever $\alpha \in \calY_d$ for some $d \in \calM_{<D}$, we have $|h_D(\alpha)| \gg \hK|d|^{-k}$. Furthermore, we have that the $\calY_d$ are pairwise disjoint and $|\calY_d| \gg |d|^k\hK^{-1}$. It follows that
    \[\int_{\calY} |h_D(\alpha)|^{1+\frac{1}{k}}\d\alpha \gg \sum_{d \in \calM_{< D}} \mu_q^2(d) |\calY_d|(\hK|d|^{-k})^{1+\frac{1}{k}} \gg \hK^{\frac{1}{k}} \sum_{d \in \calM_{<D}} |d|^{-1} \gg \hK^{\frac{1}{k}} \log(\hK). \] 
    Combining this with \eqref{crit-lower-decomp-eq-1} and the bounds obtained for the other integrals, we have the result.
\end{proof}

\begin{lemma}\label{crit-upper-lemma}
    Let $\epsilon>0$. For $K > (\frac{1}{2}+\epsilon) N$, one has that
    \[\int_\T |R_k(\alpha)|^{1+\frac{1}{k}} \d\alpha \ll \hK^{\frac{1}{k}}\log(\hK).\]
\end{lemma}
\begin{proof}
    As in Lemma \ref{subcrit-upper-lemma}, we may dissect the integral and apply H\"older's inequality and Lemma \ref{T_i,H_i-L2-bd-lemma} to see that for any $\epsilon'>0$, we have
    \begin{align*}
        \int_\T |R_k(\alpha)|^{1+\frac{1}{k}} \d\alpha &\ll \int_\T |h_D(\alpha)|^{1+\frac{1}{k}} \d\alpha + \bigg(\int_\T |H_D(\alpha)|^2 \d\alpha\bigg)^{\frac{k+1}{2k}} \\
        &\ll \int_\T |h_D(\alpha)|^{1+\frac{1}{k}} \d\alpha + (\hK^{\frac{2}{k+1}} + \hN^{\epsilon'} \hK^{\frac{1}{k}} + \hN^{\frac{1}{k+1}+\epsilon'})^{\frac{k+1}{2k}}.
    \end{align*}
    Since $K > (\frac{1}{2}+\epsilon)N$, it is easy to see that we may choose $\epsilon'$ small enough so that this becomes
    \<\label{crit-upper-decomp-eq-1}
    \int_\T |R_k(\alpha)|^{1+\frac{1}{k}} \d\alpha \ll \int_\T |h_D(\alpha)|^{1+\frac{1}{k}} \d\alpha + O(\hK^{\frac{1}{k}}).
    \>

    Now considering the integral over $h_D(\alpha)$, we write 
    \<\label{crit-upper-decomp-eq-2}
        \int_\T |h_D(\alpha)|^{1+\frac{1}{k}} \d\alpha \ll N \max_{i<D} \int_\T |T_i(\alpha)|\bigg|h_i(\alpha)+\sum_{i<j<D} T_j(\alpha)\bigg|^{\frac{1}{k}} \d\alpha \ll N \max_{i<D} (V_1(i)+V_2(i))
    \>
    where
    \[V_1(i) = \int_\T |T_i(\alpha)||h_i(\alpha)|^{\frac{1}{k}} \d\alpha, \qquad V_2(i) = \int_\T |T_i(\alpha)|\bigg|\sum_{i<j<D} T_j(\alpha)\bigg|^{\frac{1}{k}} \d\alpha.\]
    Applying H\"older's inequality, Lemma \ref{T_i,H_i-L2-bd-lemma}, and Corollaries \ref{T_i-Ls-bd-corollary} and \ref{h_i-Ls-bd-corollary} to each of these, we see that
    \[V_1(i) \ll \bigg(\int_\T |T_i(\alpha)|^{\frac{k}{k-1}}\bigg)^{1-\frac{1}{k}}\bigg(\int_\T|h_i(\alpha)|\d\alpha\bigg)^{\frac{1}{k}} \ll (\hK^{\frac{1}{k}}\hi^{-\frac{1}{k}})(\hi^{\frac{1}{k}}) \ll \hK^{\frac{1}{k}}\]
    and
    \[V_2(i) \ll \bigg(\int_\T |T_i(\alpha)|^{\frac{2k}{2k-1}}d\alpha\bigg)^{1-\frac{1}{2k}}\bigg(\int_\T\bigg|\sum_{i<j<D} T_j(\alpha)\bigg|^{2}d\alpha\bigg)^{\frac{1}{2k}} \ll (\hK^{\frac{1}{2k}}\hi^{\frac{1}{2}-\frac{1}{2k}})(\hK^{\frac{1}{2k}}\hi^{\frac{1}{2k}-\frac{1}{2}}) \ll \hK^{\frac{1}{k}}.\]
    (Here we make use of the observation that the bounds on $T_i(\alpha)$ and $h_i(\alpha)$ simplify when $i<\hD$ and $K > (\frac{1}{2}+\epsilon) N$.) Thus upon returning to \eqref{crit-upper-decomp-eq-1} and \eqref{crit-upper-decomp-eq-2}, we have the desired result.
\end{proof}

\subsection{The Case $s>1+\frac{1}{k}$}

For convenience, we define
\[
\gamma_{k,s} = \begin{cases}
    \frac{1}{2(k+1)}(1+\frac{1}{s-1})\quad &\text{if } 1+\frac{1}{k}<s < 2, \\
    \frac{1}{k+1} \quad &\text{if } s\geq 2.
\end{cases}
\]

\begin{lemma}
    Let $s>1+\frac{1}{k}$ and $\epsilon>0$, and suppose $K > (\gamma_{k,s}+\epsilon)N$. Then one has that
    \[\int_\T |R_k(\alpha)|^s \d\alpha \asymp \hK^{s-1}.\]
\end{lemma}
\begin{proof}
    First we consider the upper bound. If $s>2$, then the desired upper bound follows from the case $s \leq 2$ and the trivial bound $|R_k(\alpha)| \leq \hK$. 
    
    Suppose $1 + \frac{1}{k}<s \leq 2$. As in the previous proof, applying the bound $|x+y|^s \ll |x|^s+|y|^s$, H\"older's inequality, and Lemma \ref{T_i,H_i-L2-bd-lemma} to see that
    \<\label{T_i-decomp-eq-supercrit-pf}
    \int_\T |S_k(\alpha;K)|^s \d\alpha \ll \int_\T |h_D(\alpha)|^s\d\alpha + \hK^{\frac{s}{k+1}} + \hN^{\frac{s}{2(k+1)}+\epsilon},
    \>
    and so with $s>1+\frac{1}{k}$ and $K > (\gamma_k(s)+\epsilon)N$ this is sufficient. For the other integral, we have that
    \<\label{supercrit-lemma-low-T_i-decomp-eq}
    \int_\T |h_D(\alpha)|^s \d\alpha = \int_\T \bigg|\sum_{i < D} i^{-1}iT_i(\alpha)\bigg|^s \d\alpha \ll \bigg(\sum_{i < D} i^{-\frac{s}{s-1}}\bigg)^{s-1} \sum_{i<D} i^s\int_\T|T_i(\alpha)|^s\d\alpha
    \>
    by H\"older's inequality. 
    
    Since $s\leq 2$, the first sum on the right in \eqref{supercrit-lemma-low-T_i-decomp-eq} is $O(1)$. For the remaining integral, we apply H\"older's inequality to see that
    \[\int_\T |T_i(\alpha)|^s \d\alpha \ll \bigg(\int_\T |T_i(\alpha)|^{1+\delta}\d\alpha\bigg)^{\frac{2-s}{1-\delta}}\bigg(\int_\T |T_i(\alpha)|^2\d\alpha\bigg)^{\frac{s-1-\delta}{1-\delta}}.\]
    Upon applying the bounds from Lemma \ref{T_i,H_i-L2-bd-lemma} and Corollary \ref{T_i-Ls-bd-corollary}, we then see that this is bounded by
    \[\hi^{\frac{2-s}{1-\delta}}\hK^{\frac{\delta(2-s)}{1-\delta}} (\hK\hi^{1-k}+\hK^{\frac{1}{k}+\epsilon} + \hN^{\frac{1}{k+1}+\epsilon})^{\frac{s-1-\delta}{1-\delta}},\]
    which may be bounded in turn by
    \<\label{supercrit-large-expansion-eq}
    \hi^{\frac{2-s}{1-\delta}}\hK^{\frac{\delta(2-s)}{1-\delta}} (\hK\hi^{1-k})^{\frac{s-1-\delta}{1-\delta}}+\hi^{\frac{2-s}{1-\delta}}\hK^{\frac{\delta(2-s)}{1-\delta}}(\hK^{\frac{1}{k}+\epsilon})^{\frac{s-1-\delta}{1-\delta}} + \hi^{\frac{2-s}{1-\delta}}\hK^{\frac{\delta(2-s)}{1-\delta}}(\hN^{\frac{1}{k+1}+\epsilon})^{\frac{s-1-\delta}{1-\delta}}.
    \>
    
    The first term is $O(\hK^{s-1}\hi^{1-\frac{k(s-1-\delta)}{1-\delta}})$. Since we may take $\delta$ to be arbitrarily small with respect to $\epsilon$ and $s>1+\frac{1}{k}$, this is $O(\hK^{s-1}\hi^{-\phi})$ for some $\phi>0$. The second term in \eqref{supercrit-large-expansion-eq} simplifies to
    \[\hK^{\frac{s-1 + \delta(k(2-s)-1)}{k(1-\delta)}}\hi^{\frac{2-s}{1-\delta}} \ll \hK^{\frac{s-1}{k}+\epsilon}\hi^{2-s}\]
    for any $\epsilon>0$, again using the fact that we may take $\delta$ arbitrarily small. Finally, the third term contributes
    \[\hN^{\frac{s-1-\delta}{(1-\delta)(k+1)}+\epsilon}\hK^{\frac{\delta(2-s)}{1-\delta}}\hi^{\frac{2-s}{1-\delta}} \ll \hN^{\frac{s-1}{k+1}+\epsilon}\hi^{2-s}.\]
    Applying these bounds to \eqref{supercrit-lemma-low-T_i-decomp-eq}, we have then that
    \begin{align*}
        \int_\T |h_D(\alpha)|^s \d\alpha &\ll \sum_{i < D} i^s(\hK^{s-1}\hi^{-\phi} + \hK^{\frac{s-1}{k}+\epsilon}\hi^{2-s} + \hN^{\frac{s-1}{k+1}+\epsilon}\hi^{2-s}) \\
        &\ll \hK^{s-1} + \hK^{\frac{s-1}{k}+\epsilon}\hK^{\frac{2-s}{k+1}} + \hN^{\frac{s-1}{k+1}+\epsilon}\hK^{\frac{2-s}{k+1}} \\
        &\ll \hK^{s-1} + \hK^{\frac{s-1+k}{k(k+1)}+\epsilon} + \hN^{\frac{s-1}{k+1}+\epsilon}\hK^{\frac{2-s}{k+1}}.
    \end{align*}
    Observe that since $1+\frac{1}{k}<s\leq 2$, the middle term is at most $\hK^{\frac{1}{k}+\epsilon} \ll \hK^{s-1}$. Therefore upon combining this bound with \eqref{T_i-decomp-eq-supercrit-pf}, we can see that
    \[\int_\T |R_k(\alpha)|^s \d\alpha \ll K^{s-1}+N^{\frac{s}{2(k+1)}+\epsilon}+N^{\frac{s-1}{k+1}+\epsilon}K^{\frac{2-s}{k+1}},\]
    and one may easily check that the latter two error terms are dominated by $\hK^{s-1}$ whenever $K > (\gamma_{k,s}+\epsilon)N$.

    Thus we have the upper bound for all $s>1+\frac{1}{k}$. For the lower bound, we focus first on the case $s=2$. For this case, we imitate the argument made in the proof of Lemma \ref{crit-lower-lemma}. Observe that, upon expanding $R_k(\alpha) = h_D(\alpha) +H_D(\alpha)$, we have that
    \<\label{s=2-lwr-eq-1}
    \int_\T |R_k(\alpha)|^2 \d\alpha - \int_\T |h_D(\alpha)|^2 \d\alpha \ll \int_\T |h_D(\alpha)H_D(\alpha)| \d\alpha + \int_\T |H_D(\alpha)|^2 \d\alpha.
    \>
    The first integral on the right hand side is $o(\hK)$ when $K > (\frac{1}{k+1}+\epsilon)N$ by Lemma \ref{hH-mixed-moment-lemma}. Meanwhile, by Lemma \ref{T_i,H_i-L2-bd-lemma}, we have that
    \[
    \int_\T |H_D(\alpha)|^2 \d\alpha \ll \hK^{\frac{2}{k+1}} + \hN^{\frac{1}{k+1}+\epsilon'},
    \]
    and choosing $\epsilon'<\epsilon$ this is $o(\hK)$ whenever $K > (\frac{1}{k+1}+\epsilon)N$. Finally, applying Lemma \ref{calY-lemma} only for $d=1$, we have that
    \[\int_\T |h_D(\alpha)|^2 \d\alpha \gg \int_{\calY_1} |h_D(\alpha)|^2 \d\alpha \gg |\calY_1|(\hK)^2 = \hK.\]
    Thus we have the result for $s=2$.
    
    The lower bound for $s \neq 2$ follows easily by convexity, though we must be slightly careful about the size constraint on $K$. Observe that $\lim_{s \rightarrow 2} \gamma_{k,s} = \frac{1}{k+1}$, and as such for each fixed $\epsilon>0$ we may choose $\delta>0$ sufficiently small that
    \[\int_\T |R_k(\alpha)|^{2-\delta} \d\alpha \ll \hK^{1-\delta}\]
    whenever $K > (\frac{1}{k+1}+\epsilon)N$. For $s>2$, we have that
    \[\hK \ll \int_\T |R_k(\alpha)|^2 \d\alpha \leq \bigg(\int_\T |R_k(\alpha)|^{2-\delta} \d\alpha\bigg)^{\frac{s-2}{s+\delta-2}}\bigg(\int_\T |R_k(\alpha)|^s\d\alpha\bigg)^{\frac{\delta}{s+\delta-2}},\]
    and so choosing $\delta>0$ sufficiently small with respect to $\epsilon$ we may deduce the lower bound for the $s$th moment in the desired range. A similar argument interpolating the $s$th moment and the $(2+\delta)$-th moment yields the lower bound for $s<2$.
\end{proof}

\section{The Middle Part Estimate}

In order to prove Lemma \ref{c_i-bd-lemma}, we must make use of the multiplicative ``short interval characters'' that populate the function field setting. One may find a more detailed exposition of these objects in the thesis of Gorodetsky \cite{gorodetsky:thesis}, but we state the essential information here.

For $f_1,f_2 \in \calM$ and an integer $M$, we say that $f_1 \equiv f_2 \mmod{R_M}$ if $f_1$ and $f_2$ share the same first $M$ next-to-leading coefficients (if $\deg(f_i)<n$, then we adopt the convention that the $n$th next-to-leading coefficient is zero). We may then define an abelian group $\calM/R_M$ with the multiplication operation, and it follows that we have an abelian group of characters $\chi$ on $\calM/R_M$, with the trivial character $\chi_0(f)$ taking the value 1 uniformly. Each character induces a completely multiplicative function $\chi^\dag$ on $\calM$ by setting $\chi^\dag(f) = \chi(\grc)$, where $\grc$ is the equivalence class of $f$ in $\calM/R_M$. We will abuse notation by writing this function as $\chi(f)$, and we will denote the group of these ``characters'' over $\calM$ by $G(R_M)$. Note that since $\calM/R_M$ is a finite abelian group, we have that $\calM/R_M \cong G(R_M)$.

In the applications considered here, such short interval characters behave just as regular Dirichlet characters. We say $S$ is a representative set of $R_M$ if it contains exactly one element from each equivalence class of $R_M$. We have the two fundamental orthogonality identities
\<\label{s.i.c.-ortho-eq-1}
    \sum_{f \in S} \chi_1(f)\overline{\chi_2}(f) = \begin{cases}
        \hA \quad &\text{if } \chi_1=\chi_2, \\
        0 \quad &\text{else}
    \end{cases}
\>
whenever $S$ is a representative set of $R_M$, and
\<\label{s.i.c.-ortho-eq-2}
    \sum_{\chi \in G(R_M)} \chi(f)\overline{\chi}(g) = \begin{cases}
        \hM \quad &\text{if } f \equiv g \mmod{R_M} \\
        0 \quad &\text{else.}
    \end{cases}
\>

On a much deeper level, Rhin \cite{rhin:1972} observed that the Riemann Hypothesis for function fields is applicable for Dirichlet series of short interval characters:
\begin{theorem}[Rhin, \cite{rhin:1972}]\label{RH-rhin-thm}
    Let $\chi \in G(R_M)$ be a non-trivial short interval character, and let $L(s,\chi)$ be its corresponding $L$-function. Then $L(s,\chi)$ is a polynomial in $q^{-s}$ of degree at most $M-1$, and furthermore $L(s,\chi)$ factors as 
    \[L(s,\chi) = \prod_{j=1}^{M-1} (1-\alpha_jq^{-s}),\]
    where $|\alpha_j| = q^{\frac{1}{2}}$ for each $j$.
\end{theorem}

As such we are able to deduce the following character sum estimate.

\begin{lemma}\label{char-sum-bd-lemma}
    Let $M \geq 1$ and $\chi \in G(R_M)$ with $\chi \neq \chi_0$. Then one has that
    \[\bigg|\sum_{f \in \calM_i} \chi(f)\bigg| \leq \hi^{\frac{1}{2}}\binom{M-1}{i} \ll \hM^\epsilon \hi^{\frac{1}{2}}.\]
\end{lemma}
\begin{proof}
    The argument is essentially due to Han \cite{han:2020}, but we repeat it here in the context of short interval characters. The case $i\geq M$ follows by the orthogonality relation \eqref{s.i.c.-ortho-eq-1} with $\chi_2=\chi_0$ (upon noting that $\calM_i$ may be partitioned into $\hi\hM^{-1}$ representative sets). Define the Dirichlet series 
    \[L(s,\chi) = \sum_{f \in \calM} \frac{\chi(f)}{|f|^s}.\]
    Let $S(R_M)$ be a representative set for the equivalence relation $R_M$. Observe that when examining the set $\calM_A$ with $A>M$, we find precisely $\hA\hM^{-1}$ polynomials in $\calM_A$ which are in any particular equivalence class modulo $R_M$. Thus we may sort by the value of $\chi(f)$ to write
    \begin{align*}
        L(s,\chi) &= \sum_{g \in \calM_{< M}} \frac{\chi(g)}{|g|^s} + \sum_{h \in S(R_M)}\chi(h)\sum_{\substack{g \in \calM_{\geq M} \\ g \equiv h \mmod{R_M}}} \frac{1}{|g|^s} \\
        &= \sum_{g \in \calM_{< M}} \frac{\chi(g)}{|g|^s} + \zeta(s)\hM^{-s}\sum_{h \in S(R_M)}\chi(h),
    \end{align*}
    where $S(R_M)$ is some representative set of the equivalence relation $R_M$. Since $\chi \neq \chi_0$, the orthogonality relation \eqref{s.i.c.-ortho-eq-1} implies that the rightmost sum equals zero. In particular, we may write
    \<\label{L-func-ident-1-eq}
    L(s,\chi) = \sum_{g \in \calM_{< M}} \frac{\chi(g)}{|g|^s} = \sum_{i< M} \bigg(\sum_{g \in \calM_i} \chi(g)\bigg)q^{-is}.
    \>

    On the other hand, Theorem \ref{RH-rhin-thm} implies that
    \<\label{L-func-ident-2-eq}
    L(s,\chi) = \prod_{\ell=1}^{M - 1} \left(1 - \alpha_\ell q^{-s}\right) = \sum_{i=0}^{M -1} \left((-1)^i\sum_{\ell_1 < \cdots < \ell_i} \prod_{b=1}^i \alpha_{\ell_b}\right)q^{-is},
    \>
    where $|\alpha_\ell|=q^{\frac{1}{2}}$ for each $\ell$. By matching coefficients in \eqref{L-func-ident-1-eq} and \eqref{L-func-ident-2-eq}, we see that
    \[\sum_{g \in \calM_i} \chi(g) = (-1)^j\sum_{\ell_1 < \cdots < \ell_i} \prod_{b=1}^i \alpha_{\ell_b},\]
    and so
    \[\bigg|\sum_{g \in \calM_i} \chi(g)\bigg| \leq \hi^{\frac{1}{2}}\sum_{\ell_1<\cdots<\ell_i} 1 = \hi^{\frac{1}{2}}\binom{M-1}{i} \ll \hM^\epsilon \hi^{\frac{1}{2}}. \qedhere\]
\end{proof}

\begin{lemma}\label{chi^k=chi_0-lemma}
    Suppose $p^c \| k$. Then there are exactly $\hM \hR^{-1}$ solutions to $\chi^k = \chi_0$ in $G(R_M)$, where
    \<\label{R-def-eq}
    R = \bigg\lfloor \dfrac{\max\{M+1-p^c,0\}}{p^c}\bigg\rfloor.
    \>
\end{lemma}
\begin{proof}
    First observe that, since $p \nmid (k/p^c)$ and $G(R_M)$ is an finite abelian $p$-group, the number of solutions to $\chi^k = \chi_0$ is the same as the number of solutions to $\chi^{p^c} = \chi_0$. Furthermore, since $G(R_M) \cong \calM/R_M$, we may instead count solutions to $f^{p^c} \equiv 1 \mmod{R_M}$ where $|f|<\hM$. 

    Writing $f = a_0 + a_1t+\cdots + a_{M-1}t^{M-1} + t^M$, we see that since we are working in characteristic $p$, we have
    \[f^{p^c} = a_0 + a_1t^{p^c} + \cdots + a_{M-1}t^{p^c(M-1)} + t^{p^cM}.\]
    Thus $f^{p^c} \equiv 1 \mmod{R_M}$ if and only if $a_{M-j} = 0$ for all $1 \leq j \leq R$, where $R$ is the greatest non-negative integer such that $Rp^c < M+1$. If $p^c \geq M+1$, then we may take $R=0$, and otherwise we may take $R= \lfloor (M+1-p^c)/p^c\rfloor$, so that $R$ is defined by \eqref{R-def-eq}. Thus we have that the number of solutions to $f^{p^c} \equiv 1 \mmod{R_M}$ with $|f|<\hM$ is exactly $\hM\hR^{-1}$, and we are done.
\end{proof}

In particular, Lemma \ref{chi^k=chi_0-lemma} implies that for $\hN$ sufficiently large, there are $O(\hM^{1-p^{-c}})$ solutions to $\chi^k = \chi_0$ in $G(R_M)$. With this, we are prepared to prove the main lemma.

\begin{proof}[Proof of Lemma \ref{c_i-bd-lemma}]
    By the definition \eqref{c_i-def-eq} of $c_i(f)$, we can see that
    \[\sum_{|f-F|<\hK} |c_i(f)|^2 \leq \sum_{d_1,d_2 \in \calM_i} \sum_{\substack{|f-F|<\hK \\ [d_1,d_2]^k|f}} 1.\]
    If $|[d_1,d_2]|^k \leq \hK$, then the value of the inner sum is exactly $\hK/|[d_1,d_2]|^k$. Otherwise, the value of the inner sum is either 0 or 1. Thus we have
    \<\label{c_i-lemma-decomp-eq-1}
    \sum_{|f-F|<\hK} |c_i(f)|^2 \leq \sum_{d_1,d_2 \in \calM_i} \dfrac{\hK}{|[d_1,d_2]|^k} + \sum_{\substack{d_1,d_2 \in \calM_Y \\ |[d_1,d_2]|^k > \hK}} \sum_{\substack{|f-F|<\hK \\ [d_1,d_2]^k |f}} 1 =: A_1+A_2.
    \>
    For the term $A_1$, we may sort by $h=(d_1,d_2)$ to write
    \[\sum_{d_1,d_2 \in \calM_i} \dfrac{\hK}{|[d_1,d_2]|^k} = \sum_{h \in \calM_{\leq i}} \sum_{\substack{d_1,d_2 \in \calM_i \\ (d_1,d_2)=h}} \dfrac{\hK}{|[d_1,d_2]|^k} \leq \sum_{h \in \calM_{\leq i}} \sum_{d_1,d_2 \in \calM_{i-\deg h}} \dfrac{\hK|h|^k}{|d_1d_2|^k} \ll \hK\hi^{1-k}.\]

    Now for the second term in \eqref{c_i-lemma-decomp-eq-1}, we also sort by $h=(d_1,d_2)$ so that we have the bound
    \[A_2 \leq \sum_{C_1 \leq H \leq C_2} \sum_{h \in \calM_H} \sum_{d_1,d_2 \in \calM_{i-H}} \sum_{|d_1^kd_2^kh^ku-g|<\hK} 1 = \sum_{C_1 \leq H \leq C_2} A_2(H),\]
    where $C_1 = \max\{0,2i-N/k\}$ and $C_2 = \min\{i,2i-K/k\}$. Observe that for any $H$ in the range of summation, there can be at most one $u$ satisfying the condition on the innermost sum. Thus we have the trivial bound
    \<\label{A_2-triv-bd}
    A_2(H) \ll \sum_{h \in \calM_H} \sum_{d_1,d_2 \in \calM_{i-H}} 1 \ll \hi^2\hH^{-1}.
    \>
    This is an insufficient bound when $H$ is small. To improve this, we represent the innermost condition in $A_2(H)$ as a character sum. Setting $M = N-K-1$, we have
    \[
        A_2(H) = \hM^{-1}\sum_{h \in \calM_H} \sum_{d_1,d_2 \in \calM_{i-H}} \sum_{u \in \calM_{N+(H-2i)k}} \sum_{\chi \in G(R_M)}\chi(d_1^kd_2^kh^ku)\overline{\chi}(F),
    \]
    and since $\chi$ is completely multiplicative this becomes
    \<\label{A_2-char-decomp-eq}
        A_2(H) = \hM^{-1} \sum_{\chi \in G(R_M)}\overline{\chi}(F) \bigg(\sum_{h \in \calM_H} \chi^k(h)\bigg)\bigg(\sum_{d \in \calM_{i-H}} \chi^k(d)\bigg)^2\bigg(\sum_{u \in \calM_{N+(H-2i)k}}\chi(u)\bigg).
    \>
    
    Now if $\chi = \chi_0$, then we have that
    \<\label{chi_0-bd-eq}
    \bigg(\sum_{h \in \calM_H} \chi^k(h)\bigg)\bigg(\sum_{d \in \calM_{i-H}} \chi^k(d)\bigg)^2\bigg(\sum_{u \in \calM_{N+(H-2i)k}}\chi(u)\bigg) \ll \hN\hi^{2(1-k)}\hH^{k-1}.
    \>
    Next, suppose that $\chi \neq \chi_0$ but $\chi^k = \chi_0$. Then the sums over $h$ and $d$ are trivial, and we may apply Lemma \ref{char-sum-bd-lemma} to the sum over $u$ to see that
    \<\label{chi^k=chi_0-bd-eq}
    \bigg(\sum_{h \in \calM_H} \chi^k(h)\bigg)\bigg(\sum_{d \in \calM_{i-H}} \chi^k(d)\bigg)^2\bigg(\sum_{u \in \calM_{N+(H-2i)k}}\chi(u)\bigg)  \ll \hN^{\frac{1}{2}+\epsilon} \hi^{2-k} \hH^{\frac{k}{2}-1}.
    \>
    If $\chi \neq \chi_0$ and $\chi^k \neq \chi_0$, we may instead apply Lemma \ref{char-sum-bd-lemma} to each of the sums so that
    \<\label{chi^k-neq-chi_0-bd-eq}
    \bigg(\sum_{h \in \calM_H} \chi^k(h)\bigg)\bigg(\sum_{d \in \calM_{i-H}} \chi^k(d)\bigg)^2\bigg(\sum_{u \in \calM_{N+(H-2i)k}}\chi(u)\bigg) \ll \hN^{\frac{1}{2}+\epsilon}\hi^{1-k}\hH^{\frac{k-1}{2}}.
    \>
    
    It is clear that only a single character requires a bound of the form \eqref{chi_0-bd-eq}. For the bound \eqref{chi^k=chi_0-bd-eq}, Lemma \ref{chi^k=chi_0-lemma} shows that there are $O(\hM^{1-p^{-c}})$ such characters, and there are trivially $O(\hM)$ characters giving the bound \eqref{chi^k-neq-chi_0-bd-eq}. Thus by \eqref{A_2-char-decomp-eq} we have that
    \<\label{A_2-bd-eq-2}
    A_2(H) \ll \hK\hi^{2(1-k)}\hH^{k-1} + \hN^{\frac{1}{2}-p^{-c}+\epsilon}\hK^{p^{-c}} \hi^{2-k} \hH^{\frac{k}{2}-1} +\hK^{\frac{1}{2}+\epsilon}\hi^{2-k}\hH^{\frac{k}{2}-1}.
    \>
    The result is obtained by determining where the bounds \eqref{A_2-triv-bd} and \eqref{A_2-bd-eq-2} are equal. Setting $L = \max\{2i-(\frac{1}{k}-\frac{2}{p^ck})N - \frac{2}{p^ck}K,2i-\frac{N}{k+1}\}$, we may apply the bounds \eqref{A_2-triv-bd} and \eqref{A_2-bd-eq-2} to see that
    \begin{align*}
        A_2 &\ll \sum_{C_1 \leq H < L} (\hK\hi^{2(1-k)}\hH^{k-1} +\hN^{\frac{1}{2}-p^{-c}+\epsilon}\hK^{p^{-c}} \hi^{2-k} \hH^{\frac{k}{2}-1} + \hK^{\frac{1}{2}+\epsilon}\hi^{2-k}\hH^{\frac{k}{2}-1}) + \sum_{L \leq H \leq C_2} \hi^2\hH^{-1} \\
        &\ll K\hi^{1-k} + \hN^{\frac{1}{k}-\frac{2}{p^ck}+\epsilon}\hK^{\frac{2}{p^ck}} + \hN^{\frac{1}{k+1}+\epsilon},
    \end{align*}
    and applying this along with the bound on $A_1$ to \eqref{c_i-lemma-decomp-eq-1} we see the desired result.
\end{proof}

\section*{Acknowledgments}

The author would like to thank his advisor Trevor Wooley for his constant support and suggestions. He is especially grateful to Ofir Gorodetsky, who was very helpful in discussions involving short interval characters. This work was partially supported by NSF grant DMS-2502625.

\bibliography{refs}

\end{document}